\documentclass[12pt]{amsproc}
\usepackage{amsmath, amsthm,amssymb}

\usepackage[pagebackref, colorlinks, linkcolor=red, citecolor=blue, urlcolor=blue, hypertexnames=true]{hyperref}
\usepackage{setspace}

\allowdisplaybreaks

\title{Character rigidity for special linear groups}
\author{Jesse Peterson}
\address{J.P., Department of Mathematics, Vanderbilt University, 1326 Stevenson Center, Nashville, TN 37240, U.S.A.}
\email{jesse.d.peterson@vanderbilt.edu}

\author{Andreas Thom}
\address{A.T., Mathematisches Institut, Universit\"at Leipzig,
PF 100920, 04009 Leipzig, Germany}
\email{andreas.thom@math.uni-leipzig.de}

\date{\today}

\newtheorem{thm}{Theorem}[section]
\newtheorem{prop}[thm]{Proposition}
\newtheorem{cor}[thm]{Corollary}
\newtheorem{lem}[thm]{Lemma}
\theoremstyle{definition}
\newtheorem{defn}[thm]{Definition}
\newtheorem{rem}[thm]{Remark}

\newcommand{\C}{{\mathbb C}}

\newcommand{\N}{{\mathbb N}}
\newcommand{\Z}{{\mathbb Z}}
\newcommand{\R}{{\mathbb R}}

\newcommand{\HH}{{\mathcal H}}

\newcommand{\actson}{{\curvearrowright}}

\begin{document}

\onehalfspace

\maketitle
\begin{abstract} In this paper we study characters on special linear groups $SL_n(R)$, where $R$ is either an infinite field or the localization of an order in a number field.
We give several applications to the theory of measure-preserving actions, operator-algebraic superrigidity, and almost homomorphisms.
\end{abstract}
\tableofcontents

\section{Introduction}

The study of class functions -- conjugate invariant functions -- on finite groups has a long and successful history, dating back to the work of Frobenius, Schur and others on character theory and the theory of finite-dimensional representations of finite groups. Every such representation preserves an inner product, i.e.\ is unitary with respect to a suitable Hilbert space structure.
A corresponding theory of unitary representations of infinite groups on infinite dimensional Hilbert spaces -- initiated by work of Weyl and von Neumann -- was based from very beginning more on the concept of positive definite functions instead of class functions. For finite groups, every class function is a difference of positive definite class functions, so that the relevance of positivity could only emerge through the extension to infinite groups.

A theory of positive definite class functions works well for compact groups and led for example to the Peter-Weyl theorem and extensions of many results from the character theory of finite groups. However, it seemed that class functions were less well suited for the study of non-compact Lie groups.
In 1950, the disinterest in positive definite class functions on non-compact Lie groups was substantiated by Segal and von Neumann \cite{segneu}, who showed that a connected semi-simple Lie group (none of whose simple constituents is compact) does not admit non-trivial measurable positive definite class functions. In order to treat non-compact simple Lie groups $G$ one has to introduce conjugation invariant distributions on $G$ which makes the theory significantly more complicated, see the extensive work of Harish-Chandra which was started in \cite{hc}.

In 1964, Thoma \cite{thoma1,thoma2} initiated the systematic study of positive definite class functions on general discrete infinite groups. 
Let us start with some general definitions.
\begin{defn} Let $\Gamma$ be a discrete group. A function $\varphi \colon \Gamma \to \mathbb C$ is called
\begin{itemize}
\item positive definite, if for all $\lambda_1,\dots,\lambda_n \in \mathbb C$ and all $g_1,\dots,g_n \in \Gamma$, we have
$$\sum_{i,j=1}^n \lambda_i \bar \lambda_j \varphi(g_j^{-1} g_i) \geq 0,$$
\item conjugation invariant or a class function, if we have $\varphi(hgh^{-1}) = \varphi(g)$ for all $g,h \in \Gamma$, and
\item a character, if it is positive definite and conjugation invariant.
\end{itemize}
\end{defn}

From now on let $\Gamma$ be a discrete group. Typically, we normalize characters so that $\varphi(e)=1$. Note that the set of normalized characters is a compact convex set in the unit ball of $\ell^{\infty}(\Gamma)$ with the weak* topology. This set is a Choquet simplex \cite{thoma1,thoma2} and thus, every character  arises as an integral of extremal characters in a {\it unique} way.
Supply of characters comes mainly from two classical sources. First of all, every finite-dimensional unitary representations $\pi \colon \Gamma \to U(\HH)$ yields a (normalized) character in the standard way
$$\chi_{\pi}(g) := \frac{{\rm tr}_{\HH}(\pi(g))}{\dim \HH}.$$
The character $\chi_{\pi}$ is extremal if and only if $\pi$ is iso-typical.
A second source comes from probability measure-preserving actions. Indeed, if $\sigma \colon \Gamma \to {\rm Aut}(X,\mu)$ defines a probability measure-preserving action of $\Gamma$ on a standard probability space $(X,\mu)$, then 
$$\varphi_{\sigma}(g) := \mu(\{ x \in X \mid \sigma(g)x=x \})$$
is a character on $\Gamma$. This character is extremal if the action is ergodic and the stabilizer subgroups are almost everywhere self-normalizing \cite{vershiknonfree}, however there are additional extremal characters of this type as well.
If $\sigma \colon \Gamma \to {\rm Aut}(X,\mu)$ is an essentially free action, then 
$$\varphi_{\sigma}(g) = \begin{cases} 1 & g=e \\ 0 & g \neq e \end{cases}.$$
This character is called the {\it regular character}, as its GNS-representation yields the left-regular representation of $\Gamma$. The regular character is extremal if and only if $\Gamma$ is a i.c.c.\ group, i.e.\ all non-trivial conjugacy classes are infinite. In particular, even though the action is ergodic, the associated character need not be extremal.

We say that a character $\varphi$ is induced from a normal subgroup $\Lambda \subset \Gamma$, if 
$\{ g \in \Gamma \mid \varphi(g) \neq 0\} \subset \Lambda.$
If $\Gamma$ has a non-trivial center $C \subset \Gamma,$ then any character on $C$ gives rise to a character on $\Gamma$ induced from $C$. The character which is constant equal to one is called the {\it constant character}.

Recently, stabilizers of non-free probability measure-preserving actions, and their associated characters, have received a lot of attention, much in part to the notion of {\it invariant random subgroups} introduced in \cite{abertIRS}. Many rigidity results in ergodic theory, such as for example in the seminal work of Stuck-Zimmer \cite{stuckzimmer}, and Nevo-Zimmer \cite{nevozimmer}, can be phrased as a rigidity result for invariant random subgroups.
We want to emphasize that some of these results are particular consequences of more general rigidity results for characters. Typically, a rigidity result for characters states, that the only extremal characters on $\Gamma$ which are not induced from the center of $\Gamma$ are the ones that arise from finite-dimensional unitary representations (maybe even factorizing through a finite quotient of $\Gamma$).

Early progress in the classification of characters on particular discrete groups has been obtained by Thoma for $S_{\infty}$ \cite{thoma2} and later by Kirillov \cite{kir} in 1966 who described the characters on the groups $GL_n(k)$ for $n \geq 2$ and $SL_n(k)$ for $n \geq 3$, where $k$ is an arbitrary infinite field. There is some confusion whether Kirillov also covered the case $SL_2(k)$, but his methods indicate that this is not the case. His study was continued and extended by Ovcinnikov \cite{ovc} in 1971 who admits more clearly that he can cover all Chevalley groups not of type $A_1$ and $B_2$. Here, $A_1$ corresponds to $SL_2$. Ovcinnikov says in a footnote that he can also prove his main result for the remaining cases. However, an argument was never published.

Later, Skudlarek studied characters on $GL_{\infty}(k)$ \cite{skud}, where $k$ is a finite field. Vershik and Kerov extended the study of characters on the infinite symmetric group \cite{vershik} and Voiculescu studied the analogous question for the topological group $U(\ell^2 \mathbb N)$ in \cite{voi}. More recently, Bekka \cite{bekka} obtained a complete classification of characters on $SL_n(\mathbb Z)$ for $n \geq 3$. His techniques can be easily adapted to give a new proof of Kirillov's result for $SL_n(k)$ for $n \geq 3$. Bekka observed that his results imply some form of operator-algebraic superrigidity, saying that there is a unique II$_1$-factor representation for $PSL_n(\Z)$ with $n \geq 3$, or equivalently, for every II$_1$-factor $N$ with $PSL_n(\Z) \subset U(N)$ and $PSL_n(\Z)''=N$, we must have that the inclusion $PSL_n(\Z) \subset U(N)$ extends to an isomorphism $L(PSL_n(\Z)) = N$. Here, $L(PSL_n(\Z))$ denotes the group von Neumann algebra of $PSL_n(\Z)$, see Section \ref{characterrigidity} for a definition. This complements rigidity results of Mostow \cite{mostow}, Margulis \cite{margulis}, and Prasad \cite{prasad}, and suggests that one should look for other irreducible lattices in higher rank groups for which operator-algebraic superrigidity holds. In a recent preprint, Dudko and Medynets \cite{dudkomedynets} have studied character rigidity for the Higman-Thompson group and related groups.

\vspace{0.2cm}

The purpose of this note is twofold. First of all, we now want to come back to the case $SL_2(k)$, which was left open by Kirillov and Ovcinnikov. One of our main results is a complete classification of the characters. 

Moreover, we prove a character rigidity result for groups $SL_2(BS^{-1})$, where $B$ is an order in a number field $K$, $S \subset B$ is a multiplicative subset, $BS^{-1}$ denotes the localization of $B$ by $S$ and we are assuming that $BS^{-1}$ has infinitely many units. It is well-known that in many situations these groups enjoy rigidity properties much like lattices in simple algebraic groups of higher rank. The groups themselves are irreducible lattices in products of almost simple algebraic groups of rank one.

The second aim of this note is derive some applications of the above character rigidity results to rigidity phenomenon for invariant random subgroups, and almost representations.  

We show that character rigidity implies a non-existence result for invariant random subgroups. This was also shown independently by Dudko and Medynets in \cite{dudkomedynets}. As a particular consequence, we can reprove and extend unpublished results by Ab\'ert, Avni and Wilson \cite{abert} about the non-existence of invariant random subgroups of $SL_n(k)$ for $n \geq 2$. More concretely, we obtain that for any infinite field, any non-trivial measure-preserving ergodic action of $PSL_2(k)$ on a probability measure space is essentially free.  We derive similar results for groups $PSL_n(k)$ and $PSL_n(BS^{-1})$ for $n \geq 3$, and for $PSL_2(BS^{-1})$ if $BS^{-1}$ contains infinitely many units.

We also derive operator-algebraic superrigidity results and relate our work to the study of almost homomorphisms. In particular, we apply our results and Ozawa's work on Property $\rm T_Q$ \cite{ozawa} to show that there are no non-trivial almost homomorphism from $PSL_n(k)$ to the unitary group of any {\rm II}$_1$-factor with the Haagerup property. Here, $n \geq 3$ and $k$ is a field which is not an algebraic extension of a finite field.

\vspace{0.2cm}

This article is organized as follows. Section \ref{characterrigidity} contains our main results on character rigidity for special linear groups over infinite fields or certain subrings of number fields. Secton \ref{oas} clarifies the relationship of character rigidity with rigidity of invariant random subgroups or equivalently fields of stabilizers of probability measure-preserving actions. It also contains some application towards a von Neumann algebraic superrigidity. In Section \ref{almostrep} we study almost representation into finite von Neumann algebras with the Haagerup property and prove a rigidity result for $PSL_n(k)$, where $n \geq 3$ and $k$ is an infinite field.

\section{Character rigidity}
\label{characterrigidity}
In this section, we want to study character rigidity of the group $SL_2(R)$, where $R$ is a commutative ring. 
\subsection{Generalities} The main tool to study characters on discrete groups is the GNS-construction, named after Gel'fand, Na\u{\i}mark and Segal. Let us recall some of its basic features. Let $\Gamma$ be a discrete group and let $\varphi \colon \Gamma \to \C$ be a character on $\Gamma$. We denote by $\C[\Gamma]$ the complex group ring of $\Gamma$, which comes equipped with a natural involution $(\sum_{g} a_g g)^*= \sum_g \bar a_{g} g^{-1}$. By abuse of notation, we denote the linear extension of $\varphi$ to $\C[\Gamma]$ also by $\varphi$. Since $\varphi$ is positive-definite, the formula $\langle a,b \rangle := \varphi(b^*a)$ defines a positive semi-definite sesqui-linear form on $\C[\Gamma]$. Via separation and completion we obtain a Hilbert space $\HH_{\varphi}$, equipped with a natural map $\Gamma \ni g \mapsto \xi_g \in \HH_{\varphi}$. We also write $\Omega:=\xi_e$. It is a standard fact that the left-action of $\Gamma$ on $\C[\Gamma]$ extends to a unitary representation $\pi \colon \Gamma \to U(\HH_{\varphi})$ and $\varphi(g)= \langle \pi(g) \Omega,\Omega \rangle$, for all $g \in \Gamma$. A natural invariant of $\varphi$ is $L(\Gamma,\varphi):= \pi(\Gamma)''$, the von Neumann algebra generated by $\pi(\Gamma)$. Here, $S' := \{x \in B(\HH_{\varphi}) \mid xy=yx, \forall y \in S \}$ as usual.

Since $\varphi$ is a character, $\varphi(x) := \langle x \Omega,\Omega \rangle$ defines a unital, faithful, and normal trace on $L(\Gamma,\varphi)$. Hence, $L(\Gamma,\varphi)$ is a finite von Neumann algebra. Now, the central decomposition of $L(\Gamma,\varphi)$ into a direct integral of factors is just a mirror image of the decomposition of $\varphi \colon \Gamma \to \C$ into extremal characters. In particular, $\varphi$ is an extremal character if and only if $L(\Gamma,\varphi)$ is a factor -- which is necessarily either of type I$_n$ for $n \in \N$ or of type II$_1$. For more details on all this and definitions of all terms please consult \cite{dixmier}.

If $\varphi$ is the regular character, i.e.\ $\varphi(g)=0$ unless $g=e$, we set $L\Gamma:=L(\Gamma,\varphi)$. The von Neumann algebra $L\Gamma$ was first defined by Murray and von Neumann in their early work on algebras of operators and is now called {\it the group von Neumann algebra} of $\Gamma$. Note that $\HH_{\varphi} = \ell^2 \Gamma$ in this case. The unitary representation associated with the regular character is denoted $\lambda \colon \Gamma \to U(\ell^2 \Gamma)$ and called the left regular representation. It is well-known that $L\Gamma$ is a factor if and only if all non-trivial conjugacy classes of $\Gamma$ are infinite. It is clear that the trivial character $\varphi \equiv 1$ yields the trivial representation of $\Gamma$.

One technical feature of von Neumann algebras with a faithful trace which we are going to use is the existence of {\it conditional expectations}, which one can think of as non-commutative analogues (in fact generalizations) of conditional expectations of essentially bounded complex-valued random variables on a probability space with respect to a sub-$\sigma$-algebra. Let $(N,\tau)$ be a von Neumann algebra and $\tau \colon N \to \C$ a unital, faithful, normal and positive trace. For every von Neumann subalgebra $L \subset N$ there exists a unique conditional expectation $E \colon N \to L$ which is unital, completely positive, faithful, trace-preserving and satisfies $\tau(xy)=\tau(xE(y))$ for all $x \in L$ and $y \in N$, see \cite{umegaki} for details. Faithfulness means in this context that $E(x)=0$ for positive $x \in N$ only if $x=0$. It follows from these properties that $E$ is a $L$-bimodule map, i.e.\ $E(xyz)=xE(y)z$ for all $x,z \in L$ and $y \in N$.

In this context a map $\Phi \colon N \to L$ is called positive if it maps positive operators to positive operators and is called completely positive, if all the linear extensions ${\rm id}_n \otimes \Phi \colon M_n(\C) \otimes N \to M_n(\C) \otimes L$ to matrices are positive. Any positive map satisfies Kadison's inequality $\Phi(x^*)\Phi(x) \leq \Phi(x^*x)$, see \cite{kadison}. Completely positive maps and Kadison's inequality will only play a role in Section \ref{almostrep}, when we study the Haagerup property for von Neumann algebras.

Let $\varphi \colon \Gamma \to \C$ be a character and $\pi \colon \Gamma \to U(\HH_{\varphi})$ the associated unitary representation generating the von Neumann algebra $L(\Gamma,\varphi)$. Let $\Lambda \subset \Gamma$ be a subgroup. Then, there is a natural identification between the group von Neumann algebra of $\Lambda$ with respect to the restriction of $\varphi$ to $\Lambda$ and the von Neumann sub-algebra of $L(\Gamma,\varphi)$ generated by $\Lambda$, i.e.
$\pi(\Lambda)'' = L(\Lambda,\varphi|_{\Lambda}).$ Moreover, the Hilbert space $\HH_{\varphi|_{\Lambda}}$ can be identified in a natural way with the closed subspace of $\HH_{\varphi}$ spanned by $\{\pi(g)\Omega \mid g \in \Lambda \}$, and under this identification the trace-preserving conditional expectation $E: L(\Gamma, \varphi) \to L(\Lambda, \varphi|_{\Lambda})$ is given as $E(x) = e_{\Lambda} x e_{\Lambda}$ where $e_{\Lambda}$ is the orthogonal projection from $\HH_{\varphi}$ onto $\HH_{\varphi|_{\Lambda}}$. For more details consult \cite{dixmier, takesaki}.

The classification of von Neumann algebra factors into types yields a first classification of extremal characters. We call an extremal character $\tau \colon \Gamma \to \C$ compact if $L(\Gamma,\tau)$ is type I$_n$ and weakly mixing if $L(G,\tau)$ is type II$_1$. Note that it is very natural to transfer more attributes from II$_1$-factors to extremal characters, such as Property (T) of Connes and Jones \cite{connesjones}, amenability, the Haagerup property \cite{choda}, primeness, etc.

It is well-known that the set of characters is closed under complex conjugation and pointwise multiplication. There is also an exteriour product, which assigns to characters $\tau_1 \colon \Gamma \to \C$ and $\tau_2 \colon \Lambda \to \C$ a character $\tau_1 \otimes \tau_2 \colon \Gamma \times \Lambda \to \C$ by the formula $(\tau_1 \otimes \tau_2)(g,h)= \tau_1(g)\tau_2(h)$.
Moreover, every extremal character on $\Gamma \times \Lambda$ is of this form.
\begin{lem} \label{productchar}
Let $\Gamma,\Lambda$ be discrete groups and let $\tau \colon \Gamma \times \Lambda \to \C$ be an extremal character. Then, 
$$\tau(gh) = \tau(g) \tau(h), \quad \forall g \in \Gamma, h \in \Lambda.$$
\end{lem}

For a proof see \cite[Satz 7]{thoma3} or \cite[Lemma 2.1]{popa83}.
%
In the sequel, we will show that certain irreducible lattices in $SL_2(\R) \times SL_2(\mathbb Q_p)$, such as $SL_2(\Z[\frac1p])$, are character rigid. The lemma above shows that even though this fails drastically (as expected) for reducible lattices, there is still a certain control in the case of products.

\subsection{The main proposition}

We denote the group of units in $R$ by $R^{\times}$. 
In order to study the characters on $SL_2(R)$ we have to define various subgroups. We set

$$
U = \left\{ \left(
\begin{array}{cc}
1 & \beta \\
0 & 1
\end{array}
\right) \mid \beta \in {R} \right\} \quad\mbox{and}\quad
D = \left\{ \left(
\begin{array}{cc}
\beta & 0 \\
0 & \beta^{-1}
\end{array}
\right) \mid \beta \in {R}^\times \right\}.
$$
The center of $SL_2(R)$ is denoted by
$$
C = \left\{ \left(
\begin{array}{cc}
\beta & 0 \\
0 & \beta^{-1}
\end{array}
\right) \mid \beta = \beta^{-1} \in {R}^\times \right\}.
$$
If $\alpha \in R^\times$ we set 
$$
U_\alpha = \left\{ \left(
\begin{array}{cc}
1 & \beta \\
0 & 1
\end{array}
\right) \mid \beta \in {(\alpha^2 - 1)R} \right\} \quad\mbox{and}\quad
D_\alpha = \left\{ \left(
\begin{array}{cc}
\alpha^{k} & 0 \\
0 & \alpha^{-k}
\end{array}
\right) \mid k \in \mathbb Z \right\}.
$$

Note that $U$ is isomorphic to $(R,+)$ and $D$ is isomorphic to the multiplicative group $(R^{\times},\cdot)$. The semi-direct product $U \rtimes D_\alpha$ -- where the matrix $ \left(
\begin{smallmatrix}
\alpha & 0 \\
0 & \alpha^{-1}
\end{smallmatrix}
\right)$ acts by multiplication with $\alpha^2$ -- 
is naturally embedded as upper triangular matrices in the group $SL_2(R)$. It is clear that $GL_2(R)$ acts by conjugation on $SL_2(R)$. If $G \subset SL_2(R)$ is a group, then we set $G^t := \{g^t \mid g \in G\}$, where $g^t$ denotes the transpose of $g$.

For any unitary representation $\pi \colon \Gamma \to U(\mathcal H)$ we say that a vector $\xi \in \mathcal H$ is {\it tracial} if the positive definite functional $\varphi(g) := \langle \pi(g) \xi,\xi \rangle$ is a character on $\Gamma$. We say that a unitary representation is a ${\rm II}_1$-factor representation, if $\pi(\Gamma)''$ is a factor of type ${\rm II}_1$.

\begin{prop} \label{propimp} 
Suppose $R$ is an abelian ring, such that $\langle U_\alpha, U_\alpha^t \rangle$ has finite index in $SL_2(R)$ for all $\alpha \in R^\times$ with $\alpha^2 \neq 1$. Let $\tau: SL_2({R}) \to \mathcal \C$ be an extremal character. Then, either: \begin{enumerate}
\item The character $\tau$ factorizes through a finite quotient of $SL_2(R)$, or
\item for each $\varepsilon > 0$ and $\alpha \in R^\times$ of infinite order, there exist $k_0$ such that for all $\beta \in R$
$$ \left| \tau \left( \left(
\begin{smallmatrix}
\alpha^{k} & \beta \\
0 & \alpha^{-k}
\end{smallmatrix}
\right) \right) \right| < \varepsilon, \quad \forall k \in \Z \colon |k| \geq k_0.$$
\end{enumerate}
Moreover, if $R$ is a field then either:\begin{enumerate}
\item The character $\tau$ is the constant character, or
\item restricted to $D$, $\tau$ is vanishing at infinity.
\end{enumerate}
\end{prop}
\begin{proof}
We suppose that $\tau$ does not factorize through a finite quotient of $SL_2(R)$ and we will show that we then have the alternative conclusion.

We denote by $\pi \colon SL_2(R) \to U(\mathcal H)$ the GNS-construction with respect to $\tau$ and set $N:= \pi(SL_2(R))''$. Denote by $\Omega \in \mathcal H$ the canonical vector such that $\tau(g)= \langle \pi(g) \Omega,\Omega \rangle$ for all $g \in SL_2(R)$. Since $\tau$ is extremal, $N$ is a finite factor and $\tau$ extends to a positive faithful trace on $N$. For $\alpha \in R^\times$ denote by $p_\alpha$ the projection onto the subspace of 
$U_\alpha$-invariant vectors. Since $U_{\alpha}$ is normalized by $U \rtimes D_{\alpha}$, we get $p_{\alpha} \in \pi(U \rtimes D_\alpha)'$. Note also that obviously $p_\alpha \in \pi(U_\alpha)''$. 

In particular, $p_\alpha \mathcal H$ and $p_\alpha^{\perp} \mathcal H$ are $(U \rtimes D_\alpha)$-invariant and we can consider the unitary representations $p_\alpha \pi_{|U \rtimes D_\alpha}$ and  $p_\alpha^{\perp} \pi_{|U \rtimes D_\alpha}$ of $U \rtimes D_\alpha$. Note that $p_{\alpha} \Omega$ and $p_{\alpha}^{\perp} \Omega$ are tracial for $U \rtimes D_{\alpha}$.

Since $p_\alpha^\perp \in \pi(U)'$ we may use Bekka's normality trick (Lemma 16 in \cite{bekka}) to show that for  
$g = \left(
\begin{smallmatrix}
\alpha^{-k} & \beta \\
0 & \alpha^k
\end{smallmatrix}
\right) \in U \rtimes D_{\alpha}$, with $k \not= 0$, and $j, l \in \mathbb N$ with $k|l$, we have \begin{equation} \label{eq1}
\tau(\pi(g) p_{\alpha^j}^\perp) = \tau(\pi(g) p_{\alpha^{l}} p_{\alpha^j}^\perp).
\end{equation} Indeed,
for all $a \in U$, we have
\begin{eqnarray*} 
\tau(\pi(g) p_{\alpha^j}^\perp) &=& \langle \pi(g) p_{\alpha^j}^\perp \Omega, p_{\alpha^j}^\perp \Omega \rangle \\ &=& \langle \pi(aga^{-1}) p_{\alpha^j}^\perp \Omega, p_{\alpha^j}^\perp \Omega \rangle \\
&=& \langle \pi(gg^{-1})\pi(aga^{-1}) p_{\alpha^j}^\perp \Omega, p_{\alpha^j}^\perp \Omega \rangle \\
&=& \langle \pi(g^{-1}aga^{-1}) p_{\alpha^j}^\perp \Omega, \pi(g^{-1}) p_{\alpha^j}^\perp \Omega \rangle.
\end{eqnarray*}
A direct computation shows $\{ g^{-1}aga^{-1} \mid a \in U \} = U_{\alpha^k}$. Note also that $U_{\alpha^l} \subseteq U_{\alpha^k}$ since $1-\alpha^k$ divides $1-\alpha^l$ if $k| l$. This implies that we can find a convex combination of the vectors $\{\pi(g^{-1}aga^{-1}) p_{\alpha^j}^\perp \Omega \mid a \in U \}$ which converges to $p_{\alpha^l }p_{\alpha^j}^\perp \Omega$. Hence, we can conclude $$\tau(\pi(g) p_{\alpha^j}^\perp) = \langle \pi(g) p_{\alpha^l} p_{\alpha^j}^\perp \Omega, p_{\alpha^j}^\perp \Omega \rangle = \tau(\pi(g) p_{\alpha^l} p_{\alpha^j}^\perp).$$

Now, if $\alpha$ has infinite order we consider the sequence of projections $( p_{\alpha^{k!}} )_k$. It is easy to see that $( p_{\alpha^{k!}} )_k$ is increasing as $k \in \mathbb N$ increases. Let us denote the supremum by $p := \sup_k p_{\alpha^{k!}}$.
We conclude from Equation \eqref{eq1}, that $\tau(\pi(g) p_{\alpha^j}^\perp) = \tau(\pi(g) p p_{\alpha^j}^\perp)$ for all $j \in \mathbb N$, $g \in (U \rtimes D_\alpha) \setminus U$. 
Taking the limit as $j$ tends to infinity, we then have $\tau(\pi(g) p^\perp) = \tau( \pi(g) p p^\perp ) = 0$, for all $g \in (U \rtimes D_\alpha) \setminus U$. 

We may also consider the transpose of $U \rtimes D_\alpha$ in $SL_2(R)$ and by symmetry if we denote by $q_\alpha$ the projection onto the space of $U_\alpha^t$-invariant vectors, and $q = \sup_k q_{\alpha^{k!}}$, then we also have $\tau(\pi(g^t) q^\perp)=0$, for all $g \in (U \rtimes D_\alpha) \setminus U$.

If we had $p \wedge q \not= 0$, then for some $k \in \mathbb N$ we would have that also $p_{\alpha^k} \wedge q_{\alpha^k} \not= 0$. Hence, $\pi$ would contain invariant vectors for $\langle U_{\alpha^k}, U_{\alpha^k}^t \rangle$. By hypothesis $\langle U_{\alpha^k}, U_{\alpha^k}^t \rangle$ is a subgroup of finite index and hence by extremality of $\tau$, the representation $\pi$ would then factor through a representation of some finite quotient of $SL_2(R)$, which we are assuming it does not. 

Thus, we have $p \wedge q = 0$, or equivalently $p^\perp \vee q^\perp = 1$. Fix $g \in D_\alpha \setminus \{ e \}$ and denote by $E$ the trace-preserving conditional expectation from $N$ to $\pi(\langle g \rangle)''$. Then for all $h \in \langle g \rangle \setminus \{ e \}$ we have
$\tau(\pi(h) E(p^\perp + q^\perp) ) = 0$. If $f$ denotes the support projection of the operator $E(p^\perp + q^\perp)$ in $\pi(\langle g \rangle)''$, then we have $E(f^\perp( p^\perp + q^\perp) f^\perp) = f^\perp E( p^\perp + q^\perp) f^\perp = 0$, and hence by faithfulness of the conditional expectation we have $f^\perp( p^\perp + q^\perp) f^\perp = 0$. Therefore $f$ cannot be smaller than the support of $p^\perp + q^\perp$ which is $p^\perp \vee q^\perp = 1$. 

Now, since $E(p^\perp + q^\perp)$ has full support it is then a cyclic vector for the representation of $\langle g \rangle$ acting on $L^2(\pi(\langle g \rangle)'', \tau) \subset \HH$. As $\tau(\pi(h) E(p^\perp + q^\perp) ) = 0$ for all $h \in \langle g \rangle \setminus \{e\}$, it then follows that this representation is the left-regular representation, and so in particular is mixing, i.e.\ the matrix coefficients $g^n \mapsto \langle \pi(g^n) \xi, \eta \rangle$ vanish at infinity for each $\xi, \eta \in L^2(\pi(\langle g \rangle)'', \tau)$.
 
We claim that $E(\pi(U)) \Omega$ is pre-compact in $L^2(\pi(\langle g \rangle)'', \tau) \subset \mathcal H$. Indeed, $U/U_{\alpha^{k!}}$ is finite and hence $p_{\alpha^{k!}} \pi(U)$ is a finite set for each $k \in \mathbb N$. In particular, $E(p_{\alpha^{k!}} \pi(U))\Omega$ is pre-compact. Since $\sup_k p_{\alpha^{k!}} = p$, we then have that $E(p \pi(U))\Omega$ is also pre-compact. Moreover, since we have $h_1h_2 \in (U \rtimes D_\alpha) \setminus U$ if $h_1 \in \langle g \rangle \setminus \{ e \}$ and $h_2 \in U$, and since $\tau(\pi(h) p^\perp) = 0$, for all $h \in (U \rtimes D_\alpha) \setminus U$, we then have that the span of $E(p^\perp \pi(U))\Omega$ is orthogonal to the subspace in $L^2(\pi(\langle g \rangle)'', \tau)$ spanned by $\{ \pi(h) \mid h \in \langle g \rangle \setminus \{ e \} \}$, which is of co-dimension at most one. Thus, the span of $E(p^\perp \pi(U))\Omega$ is at most one dimensional and it follows that $E(\pi(U))\Omega \subset L^2(\pi(\langle g \rangle)'', \tau)$ is pre-compact as claimed.

Since the representation of $\langle g \rangle$ on  $L^2(\pi(\langle g \rangle)'', \tau)$ is mixing, and since $E(\pi(U))\Omega$ is pre-compact, it follows that for each $\varepsilon > 0$ there exists $k_0 \in \mathbb N$ such that for $k \geq k_0$ we have $| \tau(\pi(g^k) x) | = | \tau(\pi(g^k) E(x)) | < \varepsilon$ for all $x \in E(\pi(U))\Omega$. In particular, there exists $k_0$ such that for all $k \in \Z$ with $|k| \geq k_0$, we have
$$
\left| \tau \left( \left(
\begin{smallmatrix}
\alpha^{k} & \beta \\
0 & \alpha^{-k}
\end{smallmatrix}
\right) \right) \right| 
= \left| \tau \left( \left(
\begin{smallmatrix}
\alpha^{k} & 0 \\
0 & \alpha^{-k}
\end{smallmatrix}
\right) \left(
\begin{smallmatrix}
1 &  \alpha^{-k}\beta\\
0 & 1
\end{smallmatrix}
\right) 
\right) \right| 
< \varepsilon.
$$

We have thus shown the proposition for the case of general rings. 

Now suppose that $R$ is a field. Then for $\alpha \in R^\times$, $\alpha^2 \not= 1$ we have $U_{\alpha}=U$. In this case we set $p:=p_{\alpha}$ and $q:=q_{\alpha}$ and proceed with the same argument as above. This time though we consider the condition expectation onto $\pi(D)''$ and see that $\tau$ restricted to $D$ is vanishing at infinity. 
\end{proof}

\subsection{The case of infinite fields}
In this section we want to prove our first main result. Before,
we need the following elementary lemma.
\begin{lem}
Let $\HH$ be a Hilbert space $\eta,\xi_1,\dots,\xi_n \in \HH$ be unit vectors such that $|\langle \xi_i,\xi_j \rangle | \leq \varepsilon$ and $\langle \eta,\xi_i \rangle = \langle \eta,\xi_j \rangle$ for all $1 \leq i < j \leq n$. If $\varepsilon \leq \frac1{2n}$, then, $$|\langle \eta, \xi_i \rangle|^2 \leq \frac1{n} + 8n \varepsilon, \quad 1 \leq i \leq n.$$ In particular, if $\varepsilon^{-1/2} \leq n \leq 2\varepsilon^{-1/2}$, then
$$|\langle \eta, \xi_i \rangle| \leq 5 \cdot \varepsilon^{1/4}, \quad \forall 1 \leq i \leq n.$$

\end{lem} \label{elementary}
\begin{proof} Set $\delta:= \langle \eta,\xi_i \rangle$. Without loss of generality $\delta >0$ and we may suppose that $\eta \in {\rm span}\{\xi_i \mid 1\leq i \leq n \}$. Then, $\eta = \sum_{i=1}^n \lambda _i \xi_i$ for some $\lambda_1,\dots,\lambda_n \in \C$. We set $\lambda := \max \{ |\lambda_i| \mid 1 \leq i \leq n \}$. It follows from
$\delta = \langle \eta,\xi_i \rangle = \sum_{j=1}^n \lambda_j \langle \xi_j,\xi_i \rangle$ that
\begin{equation} \label{firstin}
|\delta - \lambda_j | \leq n \lambda \varepsilon, \quad \forall 1 \leq j \leq n.
\end{equation}
In particular,
$\lambda \leq \delta + n \lambda \varepsilon $ and hence $\lambda \leq \delta\cdot (1 - n \varepsilon)^{-1} \leq (1 - n \varepsilon)^{-1} \leq 2$.
If $\delta- n\lambda \varepsilon \leq 0$, we get $\delta \leq 2n\varepsilon$ and we are done. So we may assume that $\delta \geq 2n\varepsilon$; then Equation \eqref{firstin} yields
\begin{equation} \label{ine}
|\lambda_i|^2 \geq (\delta- n\lambda \varepsilon)^2 \geq \delta^2 - 2n \lambda \varepsilon.\end{equation}
Now, we have
$$1 = \|\eta\|^2 = \sum_{i,j=1}^n \lambda_i \bar\lambda_j \langle \xi_i,\xi_j \rangle.$$
and hence
$$1 + 4n^2 \varepsilon \geq 1 + n^2 \lambda^2 \varepsilon \geq \sum_{i=1}^n |\lambda_i|^2 \stackrel{\eqref{ine}}{\geq}  n \delta^2 - 2 n^2 \lambda \varepsilon \geq n \delta^2 - 4n^2 \varepsilon.$$ We conclude
$\delta^2  \leq   \frac{1}{n} + 8n \varepsilon.$ This finishes the proof.\end{proof}

\begin{thm} \label{rigid}
Let $k$ be an infinite field. Then, any non-trivial extremal character on $SL_2(k)$ is induced from $C \subset SL_2(k)$.
\end{thm}
\begin{proof}
Let $\tau \colon SL_2(k) \to \C$ be a non-trivial extremal character.
First of all, if $k$ is infinite, then $SL_2(k)$ does not have any finite quotients. Any non-central element $g \in SL_2(k)$ is conjugated by some $h \in GL_2(k)$ to an element of the form $\left(
\begin{smallmatrix}
0 & 1 \\
-1 & \beta
\end{smallmatrix}
\right)$ for some $\beta \in k$. Indeed, just pick some vector $v \in k^2$ which is not scaled by $g$ and take $\{v,-g(v)\}$ as the new basis. Hence, it is enough to show that $\tau \left(\left(
\begin{smallmatrix}
0 & 1 \\
-1 & \beta
\end{smallmatrix}
\right) \right)=0$ for all $\beta \in k$, since we can replace $\tau$ by $\tau \circ {\rm Ad}(h^{-1})$. It is clear that $\langle U, U^t \rangle = SL_2(k)$, so that we can apply Proposition \ref{propimp}.
We compute 
\begin{equation} \label{eq2}
\tau  \left(\left(
\begin{smallmatrix}
0 & 1 \\
-1 & \beta
\end{smallmatrix}
\right) \right) 
 = \tau \left( \left(
\begin{smallmatrix}
\alpha & 0 \\
0 & \alpha^{-1}
\end{smallmatrix}
\right) \left(
\begin{smallmatrix}
0 & 1 \\
-1 & \beta
\end{smallmatrix}
\right) \left(
\begin{smallmatrix}
\alpha^{-1} & 0 \\
0 & \alpha
\end{smallmatrix}
\right) \right) 
 = \tau \left( \left(
\begin{smallmatrix}
0 & 1 \\
-1 & 0
\end{smallmatrix}
\right) \left(
\begin{smallmatrix}
\alpha^{-2} & -\beta \\
0 & \alpha^2
\end{smallmatrix}
\right) \right) 
\end{equation}
for all $\alpha \in k^{\times}$. Let $\Sigma \subset k^{\times}$ be an infinite subset such that $\alpha^4 \neq 1$ for all $\alpha \in \Sigma$ and that for each pair of elements $\alpha,\alpha' \in \Sigma$, $\alpha^{4} \neq \alpha'^4$.
Let $\alpha,\alpha' \in \Sigma$ be a pair of elements. Then, 
we compute \begin{equation} \label{eq3}\left(
\begin{smallmatrix}
\alpha^{-2} & \beta \\
0 & \alpha^2
\end{smallmatrix}
\right)^{-1} \left(
\begin{smallmatrix}
\alpha'^{-2} & \beta \\
0 & \alpha'^2
\end{smallmatrix}
\right) = \left(
\begin{smallmatrix}
\alpha^2\alpha'^{-2} & \beta(\alpha^2-\alpha'^2) \\
0 & \alpha^{-2}\alpha'^2
\end{smallmatrix}
\right).\end{equation}
Note that if $\alpha^4 \neq \alpha'^4$, then we have
\begin{equation}\label{eqconj}
\left(
\begin{smallmatrix}
1 & \gamma \\
0 & 1
\end{smallmatrix}
\right)\left(
\begin{smallmatrix}
\alpha^2\alpha'^{-2} & 0 \\
0 & \alpha^{-2}\alpha'^2
\end{smallmatrix}
\right)\left(
\begin{smallmatrix}
1 & -\gamma \\
0 & 1
\end{smallmatrix}
\right)=\left(
\begin{smallmatrix}
\alpha^2\alpha'^{-2} & \beta(\alpha^2-\alpha'^2) \\
0 & \alpha^{-2}\alpha'^2
\end{smallmatrix}
\right)\end{equation} 
with $\gamma = -\beta(\alpha \alpha')^{2}(\alpha^2 + \alpha'^2)^{-1}$.

By the previous proposition, with $\alpha, \alpha' \in \Sigma$, as $\alpha^2\alpha'^{-2}$ tends to infinity in $k^\times$ we have $$\tau\left(\left(
\begin{smallmatrix}
\alpha^2\alpha'^{-2} & \beta(\alpha^2-\alpha'^2) \\
0 & \alpha^{-2}\alpha'^2
\end{smallmatrix}
\right) \right) = 
\tau\left(\left(
\begin{smallmatrix}
\alpha^2\alpha'^{-2} & 0 \\
0 & \alpha^{-2}\alpha'^2
\end{smallmatrix}
\right) \right) \to
0.$$ 
Equation \eqref{eq3} then implies that the elements in the set $$\left\{ \left(
\begin{smallmatrix}
\alpha^{-2} & \beta \\
0 & \alpha^2
\end{smallmatrix}
\right) \mid \alpha \in \Sigma \right\}$$ are asymptotically orthogonal with respect to $\tau$. By Lemma \ref{elementary}, this shows that the right side of Equation \eqref{eq2} is equal to zero. This proves the claim.
\end{proof}

\begin{rem}
Clearly, for the conclusion of Theorem \ref{rigid} it is necessary to assume that $k$ is infinite. However, there are quantitative results even if $k$ is finite. Indeed, if $k$ is a finite field, $m \geq 2$, and $\chi \colon PSL_m(k) \to \C$ is any non-constant irreducible character, then Gluck \cite{gluck} showed that
$$|\chi(g)| \leq C \cdot |k|^{-1/2} \cdot \chi(e)$$
for some universal constant $C>0$ and all non-trivial $g \in PSL_m(k)$.
\end{rem}

\subsection{Characters on $SL_2(BS^{-1})$}
\label{charactersonrings}
The objective of this section is to show a rigidity result for characters on the group $SL_2(BS^{-1})$, where we follow the notation of Morris \cite{witte} and denote by $B$ an order in the ring of integers of a number field $K$, and by $S \subset B$ a multiplicative subset. Now, $BS^{-1}$ denotes the localization of $B$ at the set $S$. Our results in this section apply when $BS^{-1}$ contains infinitely many units. The prototypical example to keep in mind is ${\mathbb Z}[\frac1p]$ for some prime $p$. 
Our main result is:

\begin{thm} \label{secondmain}
Let $B$ an order in the ring of algebraic integers of some number field. Let $S \subset B$ be a multiplicative subset and $BS^{-1}$ the corresponding localization. Assume that $BS^{-1}$ has infinitely many units. Let $\tau: SL_2(BS^{-1}) \to \mathbb C$ be an extremal character. Then either 
\begin{enumerate}
\item the character $\tau$ factorizes through a finite quotient of $SL_2(BS^{-1})$, or 
\item $\tau$ is induced from $C \subset SL_2(BS^{-1})$.
\end{enumerate}
\end{thm}

Recall that a unitary representation $\pi \colon \Gamma \to U(\HH)$ is called weakly mixing, if for any finite set of vectors $F \subset \HH$ and all $\varepsilon>0$, there exists some $g \in \Gamma$, such that $\langle \pi(g)\xi,\eta \rangle < \varepsilon$ for all $\xi,\eta \in F$. We call a character $\tau \colon \Gamma \to \C$ weakly mixing if the GNS-representation associated with $\tau$ is weakly mixing. It is easy to see that a character $\tau$ is weakly mixing, if for any finite set $F \subset \Gamma$ and $\varepsilon>0$, there exists $g \in \Gamma$, such that $|\tau(gh)|< \varepsilon$ for all $h \in F$.

It is a well-known result of Dye \cite{dye} that a unitary representation is weakly mixing if and only if it does not contain any finite-dimensional sub-representation, and in particular, weak mixingness passes to finite index subgroups. Note that the only if-part is obvious. In order to see the converse implication, let $F \subset \HH$ be a finite set of vectors, such that for any $g \in \Gamma$, there exist $\xi,\eta \in F$, such that $\langle \pi(g)\xi,\eta \rangle \geq \varepsilon$. Let $p \in B(\HH)$ be the orthogonal projection onto the span of $F$. It follows that $${\rm tr}( \pi(g) p \pi(g)^* p) \geq \delta$$ for some constant $\delta>0$.
Taking the barycenter of the set $$\{ \pi(g) p \pi(g)^* \mid g \in \Gamma \}$$ inside the space of Hilbert-Schmidt operators, we obtain a Hilbert-Schmidt operator $q$, which is $\Gamma$-invariant and ${\rm tr}(qp) \geq \delta$. In particular, we get $q\neq 0$. Any eigenspace of $q$ will provide a finite-dimensional sub-representation.

In order to apply Proposition \ref{propimp} to the ring $BS^{-1}$, we need the following lemma.

\begin{lem} \label{lemwitte}
Let $BS^{-1}$ be as above and assume that $BS^{-1}$ contains infinitely many units.
Then, the subgroup $\langle U_{\alpha}, U^t_{\alpha} \rangle \subset SL_2(BS^{-1})$ is of finite index for every $\alpha \in BS^{-1}$ with $\alpha^2 \neq 1$.
 \end{lem}
\begin{proof}
Following work of Carter-Keller-Paige, Morris 
showed in \cite[Lemma 6.9]{witte} that if $\alpha^2 \neq 1$, then $\langle U_{\alpha}, U^t_{\alpha} \rangle =: {\rm E}(2,BS^{-1}, \langle \alpha^2-1 \rangle )$ contains the subgroup ${\rm E}^{\lhd}(2,BS^{-1},\mathfrak q')$ for some non-zero ideal ${\mathfrak q'} \subset BS^{-1}$, see the terminology of \cite{witte}. Now,  \cite[Theorem 6.1(1)]{witte} implies that ${\rm E}^{\lhd}(2,BS^{-1},\mathfrak q') \subset SL_2(BS^{-1})$ is a subgroup of finite index. Hence, the inclusion $\langle U_{\alpha}, U^t_{\alpha} \rangle \subset SL_2(BS^{-1})$ has finite index for every $\alpha \in BS^{-1}$ with $\alpha^2 \neq 1$.
\end{proof}

\begin{prop} \label{weakmix}
Let $BS^{-1}$ be as above and assume that $BS^{-1}$ contains infinitely many units. Let $$\tau: SL_2(BS^{-1}) \to \mathbb C$$ be an extremal character. Then, either the character $\tau$ factorizes through a finite quotient of $SL_2(BS^{-1})$, or $\tau$ is weakly mixing.
\end{prop}
\begin{proof} Let $\pi \colon SL_2(BS^{-1}) \to U(\HH)$ be the GNS-representation associated with $\tau$. We have argued above that $\tau$ is weakly mixing if and only if $\pi$ does not have any finite-dimensional sub-representation. Thus, by extremality, we may assume that the dimension of $\HH$ is finite. We set $d: = \dim \HH$. It is another well-known (but non-trivial) fact that if $BS^{-1}$ contains infinitely many units, then any finite-dimensional unitary representation of $SL_2(BS^{-1})$ factorizes through a finite quotient. This follows easily from Lemma \ref{lemwitte} and Proposition \ref{propimp}, or one can proceed more directly as follows: The restriction of $\pi$ to $B = \left\{ \left(
\begin{smallmatrix}
1 & \beta \\
0 & 1
\end{smallmatrix}
\right) \mid \beta \in B \right\} \subset SL_2(BS^{-1})$ is diagonalisable and we may assume that 
$$\tau(\left(
\begin{smallmatrix}
1 & \beta \\
0 & 1
\end{smallmatrix}
\right)) = \frac1d \sum_{k=1}^d \varphi_i(\beta), \quad \forall \beta \in B,$$
for some $\varphi_1,\dots, \varphi_d \in {\rm Hom}(B,\mathbb T)$. Note that this sum decomposition is unique up to permutation of summands.
Because $BS^{-1}$ contains infinitely many units, there is some unit $\alpha \in BS^{-1}$ which is not a root of unity, as $BS^{-1}$ is contained in a finite extension $K$ of ${\mathbb Q}$. It follows from conjugation invariance that $$\tau\left(\left(
\begin{smallmatrix}
1 & \alpha^2\beta \\
0 & 1
\end{smallmatrix}
\right)\right)= \tau\left(\left(
\begin{smallmatrix}
\alpha & 0 \\
0 & \alpha^{-1}
\end{smallmatrix}
\right)\left(
\begin{smallmatrix}
1 & \beta \\
0 & 1
\end{smallmatrix}
\right)\left(
\begin{smallmatrix}
\alpha^{-1} & 0 \\
0 & \alpha
\end{smallmatrix}
\right)\right) =\tau(\left(
\begin{smallmatrix}
1 & \beta \\
0 & 1
\end{smallmatrix}
\right)),\quad \forall \beta \in B.$$ This can only happen if $\varphi_k(\beta) = \varphi_{\sigma(k)}(\alpha^2\beta)$, for some permutation $\sigma \in S_d$, all $\beta \in B$, and all $1 \leq k \leq d$. Thus, $\varphi_k(\beta) = \varphi_k(\alpha^{2d!}\beta)$ for all $\beta \in B$ and all $1 \leq k \leq d$. In this situation, we conclude that 
$$\pi\left(\left(
\begin{smallmatrix}
1 & \alpha^{2d!}-1 \\
0 & 1
\end{smallmatrix}
\right) \right) = {\rm id}_{\HH}.$$
Now, the normal subgroup generated by $\left(
\begin{smallmatrix}
1 & \alpha^{2d!}-1 \\
0 & 1
\end{smallmatrix}
\right)$ has finite index in the group $SL_2(BS^{-1})$, see for example \cite[Theorem 5.13]{witte}.
\end{proof}

\begin{proof}[Proof of Theorem \ref{secondmain}:]
It is enough to consider the case when $S$ is finite as the general result then follows by considering inductive limits. 
Note again that if $BS^{-1}$ contains infinitely many units, it must contain a unit of infinite order, say $\alpha \in BS^{-1}$.
We will assume that we are not in the first case, i.e.\ that $\tau$ does not factor through a finite quotient. The same argument in Theorem~\ref{rigid}, now using the full strength of Proposition \ref{propimp} and Lemma \ref{lemwitte}, shows that if 
$g = \left(
\begin{smallmatrix}
0 & 1 \\
-1 & \beta
\end{smallmatrix}
\right)$ for some $\beta \in BS^{-1}$, then we have $\tau(g) = 0$. The only subtlety is that the $\gamma$ in Equation \eqref{eqconj} may not exist in $SL_2(BS^{-1})$, thus one needs to use the full strength of Proposition \ref{propimp} to see that the sequence $\left(
\begin{smallmatrix}
\alpha^{-2n} & -\beta \\
0 & \alpha^{2n}
\end{smallmatrix}
\right)$ contains large sets of asymptotically orthogonal elements with respect to $\tau$. Indeed, for given $\varepsilon>0$, Proposition \ref{propimp} shows that there exists $k_0$ such that $$\left|\tau\left(\left(
\begin{smallmatrix}
\alpha^{2(n-m)k_0} & -(\alpha^{2nk_0} - \alpha^{2mk_0})\beta \\
0 & \alpha^{-2(n-m)k_0}
\end{smallmatrix}
\right) \right)  \right| < \varepsilon, \quad  \forall \beta \in R, n \neq m.$$ Hence, the set 
$\left\{\left(
\begin{smallmatrix}
\alpha^{-2nk_0} & -\beta \\
0 & \alpha^{2nk_0}
\end{smallmatrix}
\right) \mid n \in \Z \right\}$
consists of pairwise $\varepsilon$-orthogonal elements with respect to $\tau$. By Equation \eqref{eq2}, we have
\begin{equation}
\tau  \left(\left(
\begin{smallmatrix}
0 & 1 \\
-1 & \beta
\end{smallmatrix}
\right) \right) 
 = \tau \left( \left(
\begin{smallmatrix}
0 & 1 \\
-1 & 0
\end{smallmatrix}
\right) \left(
\begin{smallmatrix}
\alpha^{-2nk_0} & -\beta \\
0 & \alpha^{2nk_0}
\end{smallmatrix}
\right) \right) 
\end{equation}
for all $n \in \Z$.
By Lemma \ref{elementary}, this is enough to conclude $|\tau(g)| = \left|\tau  \left(\left(
\begin{smallmatrix}
0 & 1 \\
-1 & \beta
\end{smallmatrix}
\right) \right) \right|\leq 5 \varepsilon^{1/4}$. Since $\varepsilon>0$ was arbitrary, it follows that $\tau(g)=0$ as before.

The study of more general $g \in SL_2(BS^{-1}) \setminus C$ needs some preparation.
Since we are not in the first case by assumption, Proposition \ref{weakmix} yields that the GNS-representation $$\pi \colon SL_2(BS^{-1}) \to U(\HH)$$ associated with $\tau$ is weakly mixing. It is clear, that the character $\tau_2(g) := |\tau(g)|^2$ is also weakly mixing and satisfies in addition $\tau_2(g) \in [0,1]$.

Let now $h \in GL_2(K)$ be some element such that $hgh^{-1} = \left(
\begin{smallmatrix}
0 & 1 \\
-1 & \beta
\end{smallmatrix}
\right)$ for some $\beta \in K$. Note that $\beta = {\rm tr}(hgh^{-1}) = {\rm tr}(g) \in BS^{-1}$ and hence $hgh^{-1} \in SL_2(BS^{-1})$. We may set $$\Gamma := SL_2(BS^{-1}) \cap hSL_2(BS^{-1})h^{-1} \subset SL_2(K)$$ and define a new character $
\tau_3 \colon \Gamma \to \C$ by the formula $\tau_3(q) := \tau_2(h^{-1}qh)$. The inclusion $\Gamma \subset SL_2(BS^{-1})$ has finite index since it contains the congruence subgroup associated with the product of all denominators occuring in $h$ and $h^{-1}$. Now, this implies that $\tau_3$ is again weakly mixing. In order to conclude, it is enough to show that $\tau_3\left(\left(
\begin{smallmatrix}
0 & 1 \\
-1 & \beta
\end{smallmatrix}
\right)\right) = 0.$

Let $\Xi \subset SL_2(BS^{-1})$ be a finite set of representatives of $SL_2(BS^{-1})/\Gamma$. We define $\tau_4 \colon SL_2(BS^{-1}) \to \C$ by the formula 
\begin{equation} \label{eqchar}
\tau_4(q) := \frac1{|\Xi|} \sum_{p \in \Xi, p^{-1}qp \in \Gamma} \tau_3(p^{-1}qp).
\end{equation} Now, $\tau_4$ is also a weakly mixing character of $SL_2(BS^{-1})$. Let $\pi_4$ be the GNS-representation associated with $\tau_4$. It follows that no sub-representation of $\pi_4$ can factor through a finite quotient.  Since $SL_2(BS^{-1})$ is finitely generated it has only countably many finite quotients. Hence, if we consider the integral decomposition of $\pi_4(SL_2(BS^{-1}))''$ into factors, we have that almost surely each factor is type {\rm II}$_1$.

By the observation from the beginning, this implies that
$\tau_4\left(\left(
\begin{smallmatrix}
0 & 1 \\
-1 & \beta
\end{smallmatrix}
\right)\right)=0$.
Since the right side of Equation \eqref{eqchar} has only non-negative summands, it follows that $\tau_3\left(\left(
\begin{smallmatrix}
0 & 1 \\
-1 & \beta
\end{smallmatrix}
\right)\right) =0$.
Finally, this shows that $\tau(g) = 0$ for any non-central element $g$ and finishes the proof.
\end{proof}

\begin{rem}
Note that if $BS^{-1}$ does not contain infinitely many units, then $S$ is trivial and $K$ is either $\mathbb Q$ or an imaginary quadratic extension of $\mathbb Q$. In either case, it was proved by Grunewald-Schwermer \cite{grunew}, that $SL_2(B)$ contains a subgroup of finite index which surjects onto a non-abelian free group. In particular, there is no hope to classify characters.
\end{rem}

\subsection{Characters on $SL_n(BS^{-1})$ for $n \geq 3$}

For convenience we also want to state an extension of the result in the higher rank case.
The techniques which were applied by Bekka \cite{bekka} to treat $SL_n(\Z)$ for $n \geq 3$, combined with the results in \cite{witte} yield the following theorem.

\begin{thm}[Bekka] \label{bekka}
Let $B$ an order in the ring of algebraic integers of some number field. Let $S \subset B$ be a multiplicative subset and $BS^{-1}$ the localization. Let $\tau: SL_n(BS^{-1}) \to \mathbb C$ be an extremal character. Then either 
\begin{enumerate}
\item the character $\tau$ factorizes through a finite quotient of $SL_n(BS^{-1})$, or 
\item $\tau$ is induced from $C \subset SL_n(BS^{-1})$.
\end{enumerate}
\end{thm}

We omit the proof since one can follow the line of arguments in \cite{bekka} almost verbatim. Since $SL_n(\Z) \subset SL_n(BS^{-1})$, one can also prove rigidity more directly assuming Bekka's result for $SL_n(\Z)$.

\section{Operator algebraic superrigidity and free actions}
\label{oas}

\subsection{Rigidity for measure-preserving actions on probability spaces}

In this section we clarify the connection between character rigidity and rigidity of measure-preserving group actions on probability measure spaces. Theorem~\ref{essfree} below was obtained independently by Dudko and Medynets in \cite{dudkomedynets}. 

\begin{prop}
Suppose $\Gamma$ is a countable group, and $\Gamma \actson (X, \nu)$ is a probability measure-preserving action such that almost every $\Gamma$-orbit is infinite. Let $\mathcal R$ denote the equivalence relation generated by this action, and let $m$ the corresponding measure on $\mathcal R$. Then the unitary representation $\Gamma \actson L^2(\mathcal R, m)$ is weakly mixing.
\end{prop}
\begin{proof}
We will show that the diagonal action $\Gamma \actson L^2(\mathcal R^2, m^2)$ is ergodic, or equivalently that there is no non-null finite measure invariant subset of $\mathcal R^2$. Note that the measure $m^2$ on $\mathcal R^2$ is given by 
$$
m^2(E) 
= \iint | E \cap \{ (\Gamma x, x) \times (\Gamma y, y) \mid x, y \in X \} |  d\nu(x) d\nu(y) \nonumber \\
$$

Suppose $E \subset \mathcal R^2$ is a positive measure invariant set. If $\Delta = \{ (x, x) \mid x \in X \}$ then $\Gamma \Delta = \mathcal R$, and hence we may assume that for some $g_0 \in \Gamma$ we have $m^2( E \cap (\Delta \times (g_0 \Delta))) > 0$. Since the action $\Gamma \actson (X, \nu)$ has infinite orbits, so does the action on $\mathcal R$, and hence so does the action on $\mathcal R^2$.
Thus, 
\begin{align}
m^2(E) 
&\geq \iint | E \cap \{ (\gamma x, x) \times (\gamma g_0 y, y) \mid \gamma \in \Gamma, x, y \in X \} | d\nu(x)d\nu(y) \nonumber \\
& = \int_{ E \cap (\Delta \times (g_0 \Delta))} | \{ (\gamma x, x) \times (\gamma g_0 y,  y) \mid \gamma \in \Gamma \} | d\nu(x)d\nu(y) \nonumber \\
& = \infty \cdot m^2( E \cap (\Delta \times (g_0 \Delta))) 
= \infty. \nonumber 
\end{align}
\end{proof}

\begin{thm} \label{essfree}
Suppose $\Gamma$ is a countable i.c.c.\ group such that the only weakly mixing character is the Dirac function $\delta_e$. Then any ergodic probability measure-preserving action of $\Gamma$ on a Lebesgue space is essentially free.
\end{thm}
\begin{proof}
Let $\Gamma$ be as above and suppose that $\Gamma \actson (X, \nu)$ is an ergodic probability measure-preserving action on a Lebesgue space. Since $(X, \nu)$ has no atoms it follows that almost every $\Gamma$-orbit is infinite, and hence from the previous proposition we have that the representation $\Gamma \actson L^2(\mathcal R, m)$ is weakly mixing. We claim that the representation $\Gamma \actson L^2(\mathcal R, m)$ generates a finite von Neumann algebra. Indeed, it generates a von Neumann subalgebra of the II$_1$-factor $L\mathcal R$ associated to the equivalence relation \cite{feldmoore}. By hypothesis it then follows that $\Gamma \actson L^2(\mathcal R, m)$ is a multiple of the left-regular representation and since $1_\Delta$ is a tracial vector we have that for each $\gamma \in \Gamma \setminus \{e\}$, 
$$
\nu({\rm Fix}(\gamma)) = \langle \gamma 1_\Delta, 1_\Delta \rangle = 0.
$$
Thus, the action $\Gamma \actson (X, \nu)$ is essentially free.
\end{proof}

Special cases of the following consequence of Theorem \ref{rigid} were obtained in a previous work by Ab\'ert-Avni-Wilson \cite{abert}. It was this work that triggered our interest in character rigidity for $SL_2(k)$.

\begin{cor} Let $m \geq 2$ and $k$ be an infinite field. Any non-trivial measure-preserving action of $PSL_m(k)$ on a probability measure space is essentially free.
\end{cor}
\begin{proof}
The case $m \geq 3$ follows directly from Kirillov's work \cite{kir} and Theorem \ref{essfree}. In the case $m=2$, it follows from Theorem \ref{essfree} and Theorem \ref{rigid}.
\end{proof}

As consequence of Theorem \ref{secondmain} and Theorem \ref{bekka}, we also obtain analogous results for the groups $SL_m(BS^{-1})$. The cases $m \geq 3$ are already covered by results of Stuck-Zimmer \cite{stuckzimmer}, Nevo-Zimmer \cite{nevozimmer}, and Creutz and the first author \cite{creutzpeterson}.

\begin{cor} Let $m \geq 2$ and let $BS^{-1}$ be a ring as explained in Section \ref{charactersonrings}. If $m=2$, assume that $BS^{-1}$ has infinitely many units. Any non-trivial ergodic measure-preserving action of $PSL_m(BS^{-1})$ on a diffuse probability measure space is essentially free.
\end{cor}

Note that the preceding results can be phrased in terms of invariant random subgroups. If $\mu$ is a conjugation invariant measure on the Chabauty space ${\rm Sub}(\Gamma)$ of subgroups of $\Gamma$ (see \cite{abertIRS} for details), then $\tau_{\mu}(g) := \mu(\{ \Lambda \in {\rm Sub}(\Gamma) \mid g \in \Lambda \})$ is the associated character. Now, the results above just say that every ergodic conjugation invariant probability measure on the space of subgroups, i.e.\ every ergodic invariant random subgroup, of either $PSL_m(k)$ or $PSL_m(BS^{-1})$ is either concentrated on the trivial group or on subgroups of finite index.
Hence, the character rigidity results that we prove are good enough to imply rigidity results for measure-preserving actions on probability spaces, or equivalently, invariant random subgroups. However, in general the character associated with an invariant random subgroup of a group $\Gamma$ does not remember the invariant random subgroup completely. This can be seen already for $\Gamma = (\Z/2\Z)^2$. Indeed, the dimension of the convex set of normalized characters on $\Gamma$ is three, whereas the dimension of the convex set of invariant random subgroups is equal to the number of subgroups minus one and hence equals four. The character associated to an invariant random subgroup remembers only the probability that a particular element is contained in the random subgroup and not the probability that some finite subset of the group is contained in the random subgroup.

Similarly, it is worth noting that there are many characters, even with values in $[0,1]$, which do not come from invariant random subgroups. For example, one can take $\varphi \colon \Z/p\Z \to [0,1]$ with $\varphi(k) := \frac14 |1 + \exp(2\pi i k/p) |^2$, for some prime $p$.

\subsection{Operator algebraic superrigidity}

As a corollary of Theorem \ref{rigid} and the results in \cite{kir}, we obtain also the following operator-algebraic superrigidity result.

\begin{thm} \label{superrigid}
Let $(M,\tau)$ be a finite von Neumann algebra, $m \geq 2$, let $k$ be an infinite field, and let $\varphi \colon PSL_m(k) \to U(M)$ be a homomorphism. Then, there exists a projection $p \in \varphi(PSL_m(k))'$ and a unital, normal, and trace-preserving embedding $\tilde{\varphi} \colon L(PSL_m(k)) \to pMp$, so that 
$$\varphi(g) = \tilde{\varphi}(\lambda(g)) + p^{\perp}, \quad \forall g \in PSL_m(k).$$
\end{thm}
\begin{proof}
We may assume that $M = \varphi(PSL_m(k))''$. Note that $g \mapsto \tau(\varphi(g))$ is a character on $PSL_m(k)$. Its decomposition into extremal characters gives rise to a direct integral decomposition of $M$ into factors. By Theorem \ref{rigid} (or the results in \cite{kir}), the only extremal characters are the trivial and the regular character. This implies that the direct integral decomposition has at most two summands, one isomorphic to $\C$ and the other isomorphic to $L(PSL_m(k))$. This finishes the proof.
\end{proof}

For special linear groups over rings $BS^{-1}$, we obtain a more refined situation.

\begin{thm} Let $m \geq 2$ and let $BS^{-1}$ be a ring as explained in Section \ref{charactersonrings}. If $m=2$, assume that $BS^{-1}$ has infinitely many units. Let $M$ be a {\rm II}$_1$-factor and set $\Gamma := PSL_m(BS^{-1})$. Let $\varphi \colon \Gamma \to U(M)$ be a homomorphism. Then, there exists a projection $p \in M$, such that $p \in \varphi(\Gamma)'$ and a trace-preserving $*$-homomorphism $$\psi \colon L\Gamma \to pMp,$$ such that $p^{\perp} \varphi(\Gamma) \subset U(M)$ is pre-compact and
$\varphi(g) = \psi(g) + p^{\perp} \varphi(g),$ for all $g \in \Gamma.$
In particular, if $\varphi(\Gamma)''=M$, then $\lambda(g) \mapsto \varphi(g)$ extends to an isomorphism $L\Gamma \cong M$.
\end{thm}
\begin{proof}
This is immediate from Theorem \ref{secondmain} and Theorem \ref{bekka} and the arguments in the proof of Theorem \ref{superrigid}.
\end{proof}

\section{Almost representations}
\label{almostrep}
In 1940, von Neumann and Wigner \cite{wigneu} showed that the group $SL_2(\mathbb Q)$ does not admit any non-trivial finite-dimensional unitary representations, (see also \cite{dye}, p.\ 128). We want to prove this result for arbitrary infinite fields and more importantly extend to almost representations into II$_1$-factors with the Haagerup property.

\begin{defn}Let $N$ be a finite factor. A map $\varphi \colon \Gamma \to U(N)$ is called a $\delta$-almost homomorphism if $\|\varphi(gh)-\varphi(g)\varphi(h)\|_2 \leq \delta$ for all $g,h \in \Gamma$.
\end{defn}

Note that this definition is different from the definition in \cite{thom3}, where the operator norm was used to measure the deviation from being a homomorphism. We say that a map $\pi \colon \Gamma \to U(N)$ is $\varepsilon$-trivial on a subset $X \subset \Gamma$, if $\|\pi(g)-1\|_2 < \varepsilon$ for all $g \in X$. Similarly, we say that two maps $\pi_1,\pi_2 \colon \Gamma \to U(N)$ are $\varepsilon$-close on $X$ if $\|\pi_1(g) - \pi_2(g)\|< \varepsilon$ for all $g \in X$.

In order to prove our next theorem, we need some preparation.
Let $m \geq 2$ be an integer. Let $(k_n)_n$ be a sequence of field with $\lim_{n \to \infty} |k_n| = \infty$, let $N_n$ be a sequence of II$_1$-factors, and let $\varphi_n \colon PSL_m(k_n) \to U(N_n)$ be a sequence of $\delta_n$-homomorphisms with $\lim_{n \to \infty} \delta_n =0$. Let $\omega \in \beta \N$ be a non-principal ultrafilter. We denote by $N^{\omega}$ the von Neumann algebraic ultraproduct of $(N_n)_n$ and by $k_{\omega}$ the usual ultraproduct of fields $(k_n)_n$. Our assumptions imply that $k_{\omega}$ is infinite. We get an induced
homomorphism
$$\varphi_{\omega} \colon PSL_m(k_{\omega}) \to N^{\omega}.$$
By Theorem \ref{superrigid}, we obtain a projection $p \in N^{\omega}$ and a trace-preserving embedding $\psi \colon L(PSL_m(k_{\omega})) \to pN^{\omega} p$, such that
\begin{equation} \label{eqstr}
\varphi_{\omega}(g) = \psi(g) + p^{\perp}, \quad \forall g \in PSL_m(k_{\omega}).
\end{equation}

\begin{thm} \label{almost}
For all $m \geq 2$, $\varepsilon>0$ there exists $\delta>0$ and $n(\varepsilon)\in \N$ such that for every field $k$ with $|k| \geq n(\varepsilon)$, every {\rm II}$_1$-factor $(N,\tau)$, and every $\delta$-homomorphism $\varphi \colon PSL_m(k) \to U(N)$ the following holds. 
There exists a projection $p \in N$ such that $\|\varphi(g)p^{\perp}-p^{\perp}\|_2 \leq \varepsilon$ and
 $|\tau(\varphi(g)p)| \leq \varepsilon$  for all $g \in PSL_m(k) \setminus \{ e \}$.
\end{thm}
\begin{proof} We argue by contradiction. Suppose that for some $\varepsilon>0$, the assertion fails for $\delta=1/n$ and $n \in \N$, for all choices of $n \in \N$. Then, there exists a sequence of $1/n$-homomorphisms $\varphi_n \colon PSL_m(k_n)  \to U(N_n)$ with $|k_n| \geq n$, such that for all projections $q_n \in N_n$, there exist group elements $g_n \in PSL_m(k_n)$ such that either $\|g_nq_n^{\perp}-q_n^{\perp}\|_2 > \varepsilon$ or $|\tau(g_nq_n)| > \varepsilon$. Now, we can use the structure of $\varphi_{\omega}$ as described in Equation \eqref{eqstr}. The projection $p \in N^{\omega}$ is represented by a sequence of projections $(p_n)_{n}$. Let $g_n$ be the corresponding element of $PSL_m(k_n)$.
Clearly, $(g_n)_n$ defines an element $g \in PSL_m(k_{\omega})$. As $n \to \omega$, we obtain either $|\tau(\varphi_{\omega}(g)p)| \geq \varepsilon$ or $\|\varphi_{\omega}(g)p^{\perp}-p^{\perp}\|_2 \geq \varepsilon$. This is a contradiction and finishes the proof.
\end{proof}

We obtain the following corollary.

\begin{cor} \label{almostcor}
For all $m \geq 2$, $\varepsilon>0$ there exists $\delta>0$ such that the following holds. For every infinite field $k$, $n \in \N$ and every $\delta$-homomorphism $\varphi \colon PSL_m(k) \to U(n)$, we have $\|1-\varphi(g)\|_2 \leq \varepsilon$ for all $g \in PSL_m(k)$.
\end{cor}

This result should be compared with some stronger results in \cite{thom3} which apply to groups $SL_n(BS^{-1})$ for $n \geq 3$ and measure almost homomorphisms in the operator norm instead of the 2-norm. The results from \cite{thom3} were extended by Gamm \cite{gamm} to cover also the case $SL_2(BS^{-1})$, when $BS^{-1}$ contains infinitely many units. Thus, the main novelty of Theorem \ref{almost} is not Corollary \ref{almostcor} but that it also applies to almost homomorphisms into an arbitrary II$_1$-factor. 

It would be interesting to decide, whether every almost homomorphism of $PSL_m(k)$ into a II$_1$-factor is either almost trivial or close to an embedding of the associated group von Neumann algebra. We can completely understand the situation in case when the II$_1$-factor is amenable or more generally, if the factor has the Haagerup property.

When $k$ is an algebraic extension of a finite field then $PSL_m(k)$ is locally finite and hence embeds into the unitary group of the amenable {\rm II}$_1$-factor $R$. On the other hand, if $k$ is not an algebraic extension of a finite field then $PSL_m(k)$ is non-amenable and so by Theorem~\ref{superrigid} it follows that any homomorphism into $U(R)$ must be the trivial homomorphism. Here we show that the conclusion of Corollary~\ref{almostcor} also holds for $\delta$-almost representations $U(N)$, when $N$ has the Haagerup property. For this we will use Property (T) for a subgroup to insure an indiscrete imbedding (cf.\ \cite{connesjones, robertson}), and hence will be valid for $m \geq 3$.
Recall the following definition from \cite[Section 3]{ozawa}.

\begin{defn} \label{tq}
A group $\Gamma$ has property $\rm T_Q$ if the following holds. For every $\varepsilon>0$, there exists $\delta>0$ and a finite subset $K \subset \Gamma$, such that for every map $\pi \colon \Gamma \to U(\HH)$ and any vector $\xi \in \HH$ such that
$$\sup \{ \|\pi(gh)\xi-\pi(g)\pi(h)\xi\| \mid g,h \in \Gamma \}< \delta, \quad \mbox{and} \quad \sup \{ \|\pi(g)\xi - \xi\| \mid g \in K \} < \delta,$$
we get
$$
\sup \{ \|\pi(g)\xi - \xi\| \mid g \in \Gamma \} < \varepsilon.$$
\end{defn}

Ozawa showed in \cite[Theorem B]{ozawa} that lattices in $PSL_m(\mathbb K)$ have Property ${\rm T_Q}$ for $m \geq 3$ and any local field $\mathbb K$, as for instance $\R$ or $\mathbb F((t))$. In particular, $PSL_m(\Z)$ and $PSL_m({\mathbb F}_p[t,t^{-1}])$ have property ${\rm T_Q}$ for every $m \geq 3$.

We also recall the Haagerup property \cite{choda} for von Neumann algebras as introduced by Connes.

\begin{defn} \label{h}
A finite factor $N$ has the Haagerup property if there exists a net of unital trace-preserving completely positive maps $\phi_\alpha \colon N \to N$, such that $\| \phi_\alpha(x) - x \|_2 \to 0$, for all $x \in N$, and each $\phi_\alpha$ is compact when viewed as an operator on $L^2(N, \tau)$.
\end{defn}

Examples of finite factors with the Haagerup property include the amenable {\rm II}$_1$-factor, free group factors \cite{haagerup}, and more generally any group von Neumann algebra of a discrete i.c.c.\ group with the Haagerup property \cite{choda}. 

\begin{prop}\label{tqhaagerup}
Suppose $\Gamma$ is a discrete group with property ${\rm T_Q}$, then for each $\varepsilon > 0$ there exists $\delta > 0$ such that if $\pi\colon \Gamma \to U(N)$ is a $\delta$-homomorphism into a finite factor $N$ with the Haagerup property, then there exists a finite subset $F \subset \Gamma$ such that $F \cdot \{ g \in \Gamma \mid \| \pi(g) - 1 \|_2 < \varepsilon \} = \Gamma$.
\end{prop}
\begin{proof}
Let $\varepsilon > 0$ be given and take $\delta > 0$ with $\delta < \varepsilon/ 3$ and $K \subset \Gamma$ as in Definition \ref{tq} for $\Gamma$, for $\varepsilon/ 3$. Let $\pi: \Gamma \to U(N)$ be a $\frac{\delta}{2}$-homomorphism and consider a net $\{ \phi_\alpha \}$ as in Definition \ref{h}. Without loss of generality, we may assume that $\pi(g^{-1}) = \pi(g)^*$, for all $g \in \Gamma$.

To each $\phi_\alpha$ there is an associated pointed $N$-bimodule $(\mathcal H_\alpha, \xi_\alpha)$ such that $\tau(\phi_\alpha(x) y) = \langle x \xi_\alpha y, \xi_\alpha \rangle$ for all $x, y \in N$ (see \cite{connescorrespondences}, \cite{popacorrespondences}, or Section 1 of \cite{popa}). We may then define a map $\sigma_\alpha \colon \Gamma \to U(\mathcal H_\alpha)$ by setting 
$$\sigma_\alpha(g) \xi := \pi(g) \xi \pi(g)^*, \quad \forall g \in \Gamma, \ \forall \xi \in \mathcal H_\alpha.$$ 

Note that for all $g, h \in \Gamma$ we have
$$
\|\sigma_\alpha(gh)\xi_\alpha - \sigma_\alpha(g) \sigma_\alpha(h)\xi_\alpha \|
\leq 2 \| \pi(gh) - \pi(g) \pi(h) \|_2
<  \delta.
$$

For each $g \in \Gamma$ we have $\| \sigma_\alpha(g) \xi_\alpha - \xi_\alpha \|^2 = 2 \Re (1 - \tau( \phi_\alpha(\pi(g)) \pi(g)^* ) ) \to 0$, hence for some fixed $\alpha$ we have $\| \sigma_\alpha(g) \xi_\alpha - \xi_\alpha \| < \delta$ for all $g \in K$. It follows from property ${\rm T_Q}$ that for this $\alpha$ we have $$\sup_{g \in \Gamma} \| \sigma_\alpha(g) \xi_\alpha - \xi_\alpha \| < \varepsilon/3.$$

Thus, using Kadison's inequality $\phi_\alpha(\pi(g)^*) \phi_\alpha(\pi(g)) \leq \phi_\alpha(\pi(g)^*\pi(g)) = 1$ we conclude 
$$
\sup_{g \in \Gamma} \| \phi_\alpha(\pi(g)) - \pi(g) \|_2^2 
\leq 2  \sup_{g \in \Gamma}\Re ( 1 - \tau( \phi_\alpha(\pi(g)) \pi(g)^*) ) 
= \sup_{g \in \Gamma}\| \sigma_\alpha(g) \xi_\alpha - \xi_\alpha \|^2 
< (\varepsilon/3)^2.
$$ 
As $\phi_\alpha$ is a compact operator on $L^2(N, \tau)$, we have $\phi_\alpha(\pi(\Gamma))$ is pre-compact in $L^2(N, \tau)$ and hence there is a finite subset $F \subset \Gamma$ such that $\inf_{h \in F} \| \pi(g) - \pi(h) \|_2 < 2\varepsilon/3$ for all $g \in \Gamma$, and hence $F \cdot \{ g \in \Gamma \mid \| \pi(g) - 1 \|_2 < \delta + 2 \varepsilon/3 \} = \Gamma$.
\end{proof}

\begin{thm} \label{almostk}
Let $k$ be a field which is not an algebraic extension of a finite field, and $m \geq 3$. For all $\varepsilon > 0$ there exists $\delta > 0$ such that if $\pi \colon PSL_m(k) \to U(N)$ is a $\delta$-almost homomorphism 
into a finite factor $N$ with the Haagerup property then $\| \pi(g) - 1 \|_2 < \varepsilon$ for all $g \in PSL_m(k)$. 
\end{thm}
\begin{proof} If the characteristic of $k$ is equal to $p$ and $k$ is not an algebraic extension of a finite field, then $\mathbb F_p[t,t^{-1}] \subset k$. Hence, we get $PSL_m(\mathbb F_p[t,t^{-1}]) \subset PSL_m(k)$. If the characteristic is equal to zero, then $PSL_m(\Z) \subset PSL_m(k)$. In either case there exists some infinite subgroup $\Gamma \subset PSL_m(k)$ which has Ozawa's Property ${\rm T_Q}$. 
Let $\varepsilon > 0$ be given. By Proposition \ref{tqhaagerup} there exists $\delta > 0$ such that if $\pi \colon PSL_m(k) \to U(N)$ is a $\delta$-almost homomorphism into a finite factor $N$ with the Haagerup property then there exists $g_0 \in \Gamma \setminus \{ e \}$ such that $\| \pi(g_0) - 1 \|_2 < \varepsilon$.

By Theorem~\ref{almost} there exists a projection $p \in N$ such that $\| \pi(g) p^\perp - p^\perp \|_2 \leq \varepsilon$, and $| \tau( \pi(g) p ) | \leq \varepsilon$ for all $g \in PSL_m(k) \setminus \{ e \}$. Since $| \tau(\pi(g_0)) | \geq 1 - 2\varepsilon$ we must have $$\tau(p^\perp) \geq |\tau(\pi(g_0)p^{\perp})| = |\tau(\pi(g_0)) - \tau(\pi(g_0)p)|  \geq 1 - 3\varepsilon$$ or equivalently $\tau(p) \leq 3 \varepsilon$. Hence, we conclude 
$$\| \pi(g) - 1 \|_2 \leq \| \pi(g)p - p \|_2 + \| \pi(g)p^{\perp}  - p^{\perp} \|_2 \leq \|\pi(g)p - p\| \cdot \tau(p)^{1/2} + \varepsilon \leq 4 \varepsilon^{1/2} + \varepsilon$$ for all $g \in PSL_m(k)$. This finishes the proof.
\end{proof}

\begin{rem}
Note that the strength of Theorem \ref{almost} is not needed in the proof of the previous theorem. Indeed, there is a universal constant $n \in \N$ (only depending on $m$) such that every element in $PSL_m(k)$ is a product of at most $n$ conjugates of any non-trivial element in $PSL_m(k)$. This shows that any $\delta$-homomorphism which is $\varepsilon$-trivial on some $g_0 \in PSL_m(k)$ is uniformly $n\varepsilon$-trivial.
\end{rem}

In order to extend the previous theorem to almost homomorphisms defined on $SL_n(BS^{-1})$ for $n \geq 3$, we need to recall some notation from \cite{witte}. Let $\mathfrak q \subset BS^{-1}$ be a non-trivial ideal. For $1 \leq i,j \leq n$ and $i \neq j$, we set $e_{ij}(a)$ to be the elementary matrix with $a$ at the $(i,j)$-coordinate and $1$ on the diagonal. We define 
$${\rm LU}(n,\mathfrak q) := \{e_{ij}(a) \mid 1 \leq i,j \leq n, i \neq j, a \in \mathfrak{q} \}.$$ We set ${\rm E}(n,\mathfrak q) := \langle {\rm LU}(n,\mathfrak q) \rangle$ denote by ${\rm LU}^{\lhd}(n,BS^{-1}; \mathfrak q)$ the set of ${\rm E}(n,BS^{-1})$-conjugates of ${\rm LU}(n,\mathfrak q)$. Now, Corollary 3.13(2) of \cite{witte} says that there exists a universal constant $r \in \N$ (only depending on the degree of $K$ over $\mathbb Q$ and on $n$), such that any element in
${\rm E}^{\lhd}(n,BS^{-1};\mathfrak q)$ is a product of at most $r$ factors from ${\rm LU}^{\lhd}(n,BS^{-1}; \mathfrak q)$.
By Theorem 3.12 of \cite{witte}, the subgroup $
{\rm E}^{\lhd}(n,BS^{-1},\mathfrak q) \subset SL_n(BS^{-1})$ has finite index.

\begin{thm}
Let $BS^{-1}$ be a ring as explained in Section \ref{charactersonrings} and let $n \geq 3$. For all $\varepsilon > 0$ there exists $\delta > 0$ such that if $\pi \colon SL_n(BS^{-1}) \to U(N)$ is a $\delta$-almost homomorphism  
into a finite factor $N$ with the Haagerup property then there exists an ideal $\mathfrak q \subset BS^{-1}$, such that $\| \pi(g) - 1 \|_2 < \varepsilon$ for all $g \in {\rm E}^{\lhd}(n,BS^{-1};\mathfrak q)$. In particular, $\pi$ is $\varepsilon$-trivial on a subgroup of finite index.
\end{thm}
\begin{proof} 
Again, we assume that $\pi(g^{-1})=\pi(g)^*$ without loss of generality. We set 
$$S(\varepsilon) := \{g \in SL_n(BS^{-1}) \mid \|1 -\pi(g)\|_2 < \varepsilon \}.$$
It is clear that $S(\varepsilon)^{-1} = S(\varepsilon)$ and $S(\varepsilon)S(\varepsilon') \subset S(\varepsilon + \varepsilon'+ \delta)$. Moreover, $g S(\varepsilon) g^{-1} \subset S(\varepsilon + 2 \delta)$ for all $g \in SL_n(BS^{-1})$. 

Let $r$ be the universal constant (only depending on the number field and on $n \in \N)$ as discussed before the statement of the theorem.
Let $\varepsilon_0 >0$ to be determined later and let us first study the restriction of $\pi$ to $SL_n(\Z) \subset SL_n(BS^{-1})$. For $\delta:=\delta(\varepsilon_0)/(14r)$ as in Proposition \ref{tqhaagerup}, we get that there exists a finite set $F \subset SL_n(\Z)$ such that
$SL_n(\Z) \subset F \cdot S(\varepsilon_0).$ Let us also consider the standard copy of $SL_2(BS^{-1}) \ltimes (BS^{-1})^2 \subset SL_3(BS^{-1})$ embedded in the first two rows. Since $\Z^2 \subset (BS^{-1})^2$ is infinite and $F$ is finite, there exists $f_0 \in F$ such that the set $\Z^2 \cap f_0 S(\varepsilon_0)$ is infinite. Hence, there exists some non-trivial element $g = \left( \begin{smallmatrix} a \\b \end{smallmatrix} \right)\in \Z^2$ with $g \in (f_0S(\varepsilon_0))^{-1} f_0S(\varepsilon_0) \subset S(2 \varepsilon_0 + \delta)$. Without loss of generality, we may assume that $b \neq 0$. The conjugation action of the group $SL_2(BS^{-1})$ on $(BS^{-1})^2$ is just the standard action and hence $\left( \begin{smallmatrix} a + cb \\ b \end{smallmatrix} \right) \in S(2 \varepsilon_0 + 3\delta)$ for all $c \in BS^{-1}$. Taking the difference with $\left(\begin{smallmatrix} a \\ b \end{smallmatrix} \right) \in S(2 \varepsilon_0 + \delta)$, we conclude that
$\left(\begin{smallmatrix} c \\ 0 \end{smallmatrix} \right) \in S(4 \varepsilon_0 + 5 \delta)$ for all $c \in BS^{-1}$ which lie in the ideal $\mathfrak q$ generated by $b \in BS^{-1}$. Hence, ${\rm LU}^{\lhd}(n,BS^{-1},q) \subset S(4 \varepsilon_0 + 7 \delta)$ and thus
$${\rm E}^{\lhd}(n,BS^{-1},\mathfrak q) \subset S(r(4 \varepsilon_0 + 8 \delta)).$$ We now choose $\varepsilon_0 = \varepsilon/(8r)$ and see that ${\rm E}^{\lhd}(n,BS^{-1},\mathfrak q) \subset S(\varepsilon)$ by our choice of $\delta>0$.
Since $
{\rm E}^{\lhd}(n,BS^{-1},\mathfrak q) \subset SL_n(BS^{-1})$ has finite index, this finishes the proof.
\end{proof}

\section{Other applications}

Here, we want to mention two applications of our results, one to almost actions on probability spaces and one to sums-of-squares decompositions in the complex group ring.

\subsection{Almost actions on probability spaces}

We denote by $[\mathcal R]$ the full group of the amenable {\rm II}$_1$-measured equivalence relation $\mathcal R$, see \cite{dyefullgroup} for details. We only state the result for $PSL_m(k)$ and leave the generalizations to $SL_m(BS^{-1})$ to the reader.

\begin{cor}
Let $k$ be a field which is not an algebraic extension of a finite field, and $m \geq 3$. For all $\varepsilon > 0$ there exists $\delta > 0$ such that if $\pi \colon PSL_m(k) \to [\mathcal R]$ is a $\delta$-almost homomorphism,  i.e.\
$$\mu\left(\{x \in X \mid \pi(gh)x \neq \pi(g)\pi(h)x \} \right) < \delta, \quad \forall  g,h \in PSL_m(k),$$ 
then $\mu\left( \{x \in X \mid \pi(g)x \neq x \}\right) < \varepsilon$ for all $g \in PSL_m(k)$. 
\end{cor}
\begin{proof}
This is an immediate corollary of Theorem \ref{almostk}.
\end{proof}

\subsection{Sums-of-squares decompositions} We want to mention some applications to a particular algebraic problem that arises in the study of group rings. It is clear that any hermitean element $a \in \C[\Gamma]$ which is a sum of hermitean squares and hermitean commutators in the complex group ring, i.e.\ which is of the form
$$a = \sum_{i=1}^n b_i^*b_i +  \sum_{j=1}^m [c^*_j,c_j]$$
for some $b_1,\dots,b_n,c_1,\dots,c_m \in \C[\Gamma]$, is sent to a non-negative value by the linear extension of every character to $\C[\Gamma]$. The cone $C \subset \C[\Gamma]_h$ formed by all sums of hermitean squares and hermitean commutators has an algebraic interiour point, see \cite{cim} for definitions. Hence, by a standard separation argument the following converse holds. If the linear extension of any character takes a non-negative value on some hermitean element $a \in \C[\Gamma]_h$, then $a + \varepsilon \in C$ for all $\varepsilon>0$. In the presence of character rigidity as for example in the case $\Gamma := PSL_m(k)$ for $m \geq 2$ and some infinite field, we obtain the following simple consequence. 

\begin{thm} \label{sos}
Let $\Gamma := PSL_m(k)$ for some infinite field and $m \geq 2$ and let $a = \sum_{g \in \Gamma} a_g g \in \C[\Gamma]$ be some hermitean element with $\sum_{g \in \Gamma} a_g >0$ and $a_e >0$. Then, the element $a \in \C[\Gamma]$ is a sum of hermitean squares and hermitean commutators.
\end{thm}

We are grateful to Tim Netzer for bringing the possibility of this application of character rigidity to our attention.

\section*{Acknowledgments} We want to thank Mikl\'os Ab\'ert for motivating discussions and sharing a preliminary version of \cite{abert} with us, and Tim Netzer for discussions about Theorem \ref{sos}. A.T.\ wants to thank Uri Bader for interesting discussions about a first version of this paper.
This work was started when J.P.\ visited Universit\"at Leipzig in September 2012, he is grateful for their hospitality. J.P.\ is supported by NSF Grant DMS-1201565, and a grant from the Alfred P. Sloan Foundation. A.T.\ is supported by ERC Grant 277728. 

We thank the unknown referee for careful proof-reading and a many helpful comments.

\end{document}